\documentclass[11pt]{amsart}
 
\usepackage{color, hyperref}
\newtheorem{theorem}{Theorem}[section]
\newtheorem{lemma}[theorem]{Lemma}
\newtheorem{proposition}[theorem]{Proposition}

\usepackage{esint} % for the \fint command
\usepackage{amsmath}

\theoremstyle{definition}
\newtheorem{definition}[theorem]{Definition}

\theoremstyle{remark}
\newtheorem{remark}[theorem]{Remark}

\numberwithin{equation}{section}

\DeclareMathOperator{\area}{area}
\DeclareMathOperator{\vol}{vol}

\DeclareMathOperator{\tr}{tr}
\DeclareMathOperator{\dive}{div}

\begin{document}

\title[Stability of extremal domains]{Stability of extremal domains for the first eigenvalue of the Laplacian operator}

\author[Cavalcante]{Marcos P. Cavalcante}
\address{
  Universidade Federal de Alagoas, Instituto de Matemática,
 Maceió, AL - 57072 - 970, Brazil}
\email{marcos@pos.mat.ufal.br}
\author[Nunes]{Ivaldo Nunes}
\address{
  Universidade Federal do Maranhão, Departamento de Matemática, 
 São Luís, MA - 65080 - 805, Brazil}
\email{ivaldo.nunes@ufma.br}

\begin{abstract}
In this paper, we compute the second variation of the first Dirichlet eigenvalue on extremal domains in general Riemannian manifolds and establish a criterion for stability. We classify the stable extremal domains in the 2-sphere and higher-dimensional spheres when the boundary is minimal. Additionally, we establish topological bounds for stable domains in a general compact Riemannian surface, assuming either nonnegative total Gaussian curvature or small volume.
\end{abstract}

\maketitle

\section{Introduction}

Let $(M^n,g)$ be a complete Riemannian manifold of dimension $n\geq2$. For each compact domain $\Omega\subset M$ with smooth boundary, let $\lambda_1(\Omega)>0$  denote the first eigenvalue of the Laplacian operator on $\Omega$ with Dirichlet boundary condition. Given $0<m<\vol(M,g)$, a very interesting isoperimetric problem in spectral geometry and shape optimization is that of investigating the minimizers or, more generally, the critical points of the functional \linebreak$\Omega\subset M\mapsto \lambda_1(\Omega)$   defined on the class of domains $\Omega\subset M$   such that $\vol(\Omega)=m$. 

Regarding this problem, Faber \cite{Faber} and Krahn \cite{Krahn} proved that, in $\mathbb{R}^n$, geodesic balls are the only global minimizers of $\Omega\subset\mathbb{R}^n\mapsto \lambda_1(\Omega)$ subject to a volume constraint. More precisely, they proved that if $\Omega\subset\mathbb{R}^n$ is a compact domain, then $\lambda_1(\Omega)\geq \lambda_1(B)$, where $B\subset\mathbb{R}^n$ is a geodesic ball such that $\vol(\Omega)=\vol(B)$, and equality holds if and only if $\Omega$ is a geodesic ball. We note that the analogous result is true  in $\mathbb{S}^n$, the round sphere, and $\mathbb{H}^n$, the hyperbolic space (see, for example, Chapter IV, Section 2, of \cite{Chavel84}).

In turn, it is known that a compact domain $\Omega_0\subset (M,g)$ with smooth boundary is a critical point of the functional $\Omega\subset M\mapsto \lambda_1(\Omega)$ with respect to volume preserving deformations of $\Omega_0$ in $M$ if and only if the first eigenfunctions of the Laplacian operator of $\Omega_0$ with Dirichlet boundary condition solve the following overdetermined  problem:
\begin{align*}\label{eigenvalueproblem}
    \left\{
    \begin{array}{rl}
        \Delta\varphi + \lambda_1(\Omega_0)\varphi = 0 & \text{in } \Omega_0,\\
        \varphi = 0 & \text{on } \partial\Omega_0,\\  
        \dfrac{\partial \varphi}{\partial\nu} = c & \text{on } \partial\Omega_0,
    \end{array}
    \right.
\end{align*}
where \(c \neq 0\) is a constant and \(\nu\) denotes the outward unit normal vector 
along \(\partial\Omega_0\). 

This follows from the first variation formula for the first Dirichlet eigenvalue of the Laplacian operator, known as the Hadamard formula, which was originally proved by Garabedian and Schiffer \cite{GarabedianSchiffer} in Euclidean space and later by El Soufi and Ilias \cite{ElsoufiIlias07} for general Riemannian manifolds. From now on, following \cite{PacardSicbaldi09}, we say that a compact domain \(\Omega_0 \subset (M,g)\) with a smooth boundary is an \textit{extremal domain} if it is a critical point of \(\Omega \subset M \mapsto \lambda_1(\Omega)\) with respect to volume preserving deformations. We refer the reader to Section \ref{basic} for details and precise definitions.

In Euclidean space \(\mathbb{R}^n\), in hyperbolic space \(\mathbb{H}^n\), and in the round hemisphere \(\mathbb{S}^n_+\), only geodesic balls are extremal domains. This result follows from the classical work of Serrin \cite{Serrin}, and its extension to space forms by Kumaresan and Prajapat \cite{KumaresanPrajapat}, on elliptic overdetermined problems, which was proved using the moving plane method introduced by Alexandrov in \cite{Alexandrov}.

Extremal domains share many similarities with embedded hypersurfaces of constant mean curvature, and significant progress has been made in recent years regarding their existence, regularity, and classification. For example, in \cite{PacardSicbaldi09}, Pacard and Sicbaldi demonstrated the existence of extremal domains with small volume in compact Riemannian manifolds \((M, g)\) by assuming that the scalar curvature of \(M\) has a nondegenerate critical point. In \cite{DelaySicbaldi15}, Delay and Sicbaldi improved this result by removing the condition on the scalar curvature of \(M\).

Recently, Lamboley and Sicbaldi \cite{LamboleySicbaldi20}  established an existence and regularity theorem for the Faber-Krahn profile of any given Riemannian manifold $(M^n,g)$, which is the function $FK:m\in(0,\vol(M))\mapsto FK(m)$ defined by
\begin{equation}
FK(m)=\inf\{\lambda_1(\Omega):\mbox{$\Omega\subset M$ is an open subset and $\vol(\Omega)=m$}\}.
\end{equation}

They proved that for any connected compact Riemannian manifold \((M, g)\) and given \(0 < m < \operatorname{vol}(M, g)\), there exists an open set \(\Omega^\star \subset M\) that is smooth except for a singular set of codimension less than \(n-5\), satisfying
\begin{equation*}\label{minimizerexistence}
FK(m)=\lambda_1(\Omega^\star) .
\end{equation*}

In particular, \(\Omega^\star\) is a smooth extremal domain if \(n = 2, 3\) or \(4\). We note that an asymptotic expansion of the Faber-Krahn profile for small volumes was obtained Druet in \cite{Druet}. On the other hand, Espinar and Mazet \cite{EspinarMazet} proved that simply connected extremal domains (and, in fact, \(f\)-extremal disks) in \(\mathbb{S}^2\) are geodesic disks.

It is worth mentioning that all the above results have parallels in the context of hypersurfaces of constant mean curvature (see, for example, \cite{Ye}, \cite{PacardXu}, \cite{Almgren}, \cite{Gruter}, \cite{GMT}, and \cite{Hopf}). For more results on the existence of extremal domains, we refer the reader to \cite{Sicbaldi10} and \cite{Sicbaldi14}.

Furthermore, numerous examples of extremal domains can be constructed on Riemannian manifolds endowed with an isometric action by a compact Lie group (see Section \ref{examples}). In particular, this construction provides many examples of extremal domains on the round sphere \(\mathbb{S}^n\) that are not geodesic balls. This shows that the moving plane method used by Serrin in \cite{Serrin} fails in this case.

In this paper, we initiate the investigation of the stability properties of extremal domains in general Riemannian manifolds. We say that an extremal domain \(\Omega \subset (M, g)\) is \textit{stable} if it minimizes the first Dirichlet eigenvalue functional up to the second order with respect to volume preserving deformations. We  point out that Shimakura studied in \cite{Shimakura2} the notions of extremability and stability for domains in the Euclidean space from a \textit{local} point of view and proved  that every \textit{local} extremal domain in $\mathbb{R}^n$ is \emph{locally} stable. An alternative proof of Shimakura's result was obtained by Jorge and de Lima in \cite{JorgedeLima}, where an index formula has been proved in this setting.

In Section \ref{SecondV}, we derive the second variation formula and establish a criterion for stability, obtaining important insights into the geometry of extremal domains. Using this criterion, we give a complete characterization of stable extremal domains in the round sphere \(\mathbb{S}^2\). We prove the following theorem:

\begin{theorem}\label{maintheorem}
Let \(\Omega \subset (\mathbb{S}^2, g_{can})\) be an extremal domain. If \(\Omega\) is stable, then it is necessarily a geodesic disk.
\end{theorem}

This theorem can be considered as the parallel of the Barbosa-do Carmo-Eschenburg theorem \cite{BdCE}. It is interesting to observe that Theorem \ref{maintheorem}, combined with Lamboley-Sicbaldi's theorem proved in \cite{LamboleySicbaldi20} and cited above, provides an alternative proof of the Faber-Krahn inequality in \(\mathbb{S}^2\), which states that for any \(0 < m < 4\pi\), if \(\Omega \subset \mathbb{S}^2\) is a compact domain with area \(m\), then:
\begin{equation*}\label{fkinequality}
\lambda_1(\Omega) \geq \lambda_1(D),
\end{equation*}
where \(D \subset \mathbb{S}^2\) is a geodesic disk with area \(v\) and equality holds if and only if \(\Omega\) is a geodesic disk.

In the case that $n=2$ and \(M^2 = (\mathbb{S}^2, g_{can})\), Theorems 1.1 and 1.2 in \cite{LamboleySicbaldi20} imply that there exists a smooth domain \(\Omega^\star\) such that \(\lambda_1(\Omega) \geq \lambda_1(\Omega^\star)\). In particular, \(\Omega^\star\) is a stable extremal domain and, in fact, a global minimizer of \(\lambda_1\) with an area constraint. By Theorem \ref{maintheorem} above, we have that \(\Omega^\star\) is a geodesic disk \(D\). Moreover, if \(\lambda_1(\Omega) = \lambda_1(D)\), then \(\Omega\) is a stable extremal domain. Applying Theorem \ref{maintheorem} again we have that \(\Omega\) is also a geodesic disk, thus proving the Faber-Krahn inequality in \(\mathbb{S}^2\).

In general, for stable extremal domains $\Omega$ in a given Riemannian surface $(M^2,g)$, we prove the following restriction on the topology of $\Omega$.

\begin{theorem}\label{topbound}
Let $(M^2,g)$ be an orientable compact Riemannian surface. If $\Omega\subset M$ is a stable extremal domain with nonnegative total Gauss curvature, that is, $\int_\Omega K_g\,da\geq 0$, then the only possible values for the genus $g_\Omega$ of $\Omega$ and the number $r$ of connected components of $\partial\Omega$ are:
\begin{enumerate}
\item $g_\Omega=0$ and $r\in\{1,2,3,4,5\}$;
\item $g_\Omega=1$ and $r\in\{1,2,3,4,5,6,7\}$.
\end{enumerate}
\end{theorem}

In particular, it follows that stable extremal domains in a sphere \((\mathbb{S}^2,g)\) with nonnegative curvature have at most \(5\) boundary components.

By using the asymptotic expansion of the Faber-Krahn profile proved by Druet in \cite{Druet} we are able to obtain the following bounds on the topology of stable extremal domains with small areas in general Riemannian surfaces, which provides, in this case, an improvement on Theorem \ref{topbound}.

\begin{theorem}\label{topbound2}
Let $(M^2,g)$ be an orientable compact Riemannian surface. There exists $\varepsilon$, depending only on the geometry of $M$, such that if $\Omega\subset M$ is a stable extremal domain with $\area(\Omega)<\varepsilon$, then the only possible values for the genus $g_\Omega$ of $\Omega$ and the number $r$ of connected components of $\partial\Omega$ are:
\begin{enumerate}
\item $g_\Omega\in\{0,2\}$ and $r\in\{1,2,3\}$;
\item $g_\Omega=1$ and $r\in\{1,2,3,4,5\}$.
\end{enumerate}
\end{theorem}

Since the Faber-Krahn inequality is true on $\mathbb{S}^n$ for any $n\geq 2$, it is a very natural question to ask if the characterization proved in Theorem \ref{maintheorem} can be extended to stable extremal domains in $\mathbb{S}^n$, for $n\geq 3$. By using a result due to Reilly (see \cite[Theorem 4]{Reilly}), we are able to give the following description of stable extremal domains $\Omega\subset\mathbb{S}^n$, $n\geq 3$, in the case $\partial\Omega$ is a minimal hypersurface.  

\begin{theorem}\label{theoremhigherdim}
Let $\Omega\subset(\mathbb{S}^n,g_{can})$ be an extremal domain such that $\partial\Omega$ is a compact minimal hypersurface. If $\Omega$ is stable, then $\Omega$ is a hemisphere.
\end{theorem}

Finally, we would like to mention that our methods are inspired by those used by Ros and Vergasta \cite{RV} to study stable constant mean curvature hypersurfaces with free boundary in convex domains of $\mathbb{R}^n$.

%SECTION:BASIC DEFINITIONs AND A FIRST VARIATION FORMULA%

\section{Basic definitions and first variation formula for $\lambda_1$}\label{basic}

Let $(M,g)$ be a smooth Riemannian manifold. For each smooth compact domain $\Omega\subset M$, let $\lambda_1(\Omega)$ denote its first eigenvalue of the Laplacian operator with Dirichlet boundary condition. We will start this section by defining what we mean by a local deformation of $\Omega$ in $M$.

A smooth \textit{local deformation} of $\Omega$  in $M$ is a smooth one-parameter family of domains $\Omega_t\subset M$, $t\in(-\varepsilon,\varepsilon)$, given by $\Omega_t=f_t(\Omega)$, where $f_t:M\to M$ is the smooth flow of some smooth vector field $V\in\mathfrak{X}(M)$. If, in addition, $\vol(\Omega_t)=\vol(\Omega)$ for all $t\in (-\varepsilon,\varepsilon)$, then we say that $\Omega_t=f_t(\Omega)$ is a \textit{volume preserving} local deformation of $\Omega$. 

Let $\Omega_t=f_t(\Omega)$, $t\in(-\varepsilon,\varepsilon)$, be a smooth local deformation of a smooth compact domain $\Omega\subset M$. As the first Dirichlet eigenvalue of the Laplacian operator  is simple, it follows from the implicit function theorem that the function $t\in(-\varepsilon,\varepsilon)\mapsto \lambda_1(\Omega_t)$ is smooth. 

\begin{definition}[Extremal domains]
We say that a smooth compact domain $\Omega\subset (M,g)$
 is \textit{extremal} if 
\begin{equation*}
\left.\dfrac{d}{dt}\right|_{t=0
}\lambda_1(\Omega_t) = 0,
\end{equation*}
for any volume preserving local deformation $\Omega_t=f_t(\Omega)$, $t\in(-\varepsilon,\epsilon)$, of $\Omega$ in $M$.
\end{definition}

Next, we will state the Hadarmad formula proved by El Soufi and Ilias \cite{ElsoufiIlias07} which gives a first variation formula for the functional $\Omega\subset M\mapsto \lambda_1(\Omega)$.

\begin{proposition}[El Soufi and Ilias, \cite{ElsoufiIlias07}]\label{1stvarformula} Let $(M,g)$ be a smooth Riemannian manifold and let $\Omega\subset M$ be a smooth compact domain. Let $\Omega_t=f_t(\Omega)$, $t\in(-\varepsilon,\varepsilon)$, be a smooth local deformation of  $\Omega$.  We have
\begin{equation}\label{hadamardformula}
\left.\dfrac{d}{dt}\right|_{t=0
}\lambda_1(\Omega_t) = - \int_{\partial\Omega} v\left(\dfrac{\partial \varphi}{\partial\nu}\right)^2 d\ell,
\end{equation}
where $\nu$ is the outward unit normal vector along $\partial\Omega$,
$v=\left\langle V,\nu\right\rangle$ is the normal displacement of $\partial\Omega$ induced by the deformation,and
$\varphi\in C^\infty(\Omega)$ is the first Dirichlet positive eigenfunction of the Laplacian operator  on $\Omega$ such that $\|\varphi\|_{L^2(\Omega)}=1$
\end{proposition}

We notice that if $\Omega_t=f_t(\Omega)$, $t\in(-\varepsilon,\varepsilon)$, is volume preserving, that is, $\vol(\Omega_t)=\vol(\Omega)$ for all $t\in(-\varepsilon,\varepsilon)$, then $\int_{\partial\Omega} v\,d\ell =0$. Conversely, if $v\in C^\infty(\partial \Omega)$ is such that $\int_{\partial\Omega}v\,d\ell=0$, then there exists a smooth local deformation $\Omega_t=f_t(\Omega)$, $t\in(-\varepsilon,\varepsilon)$, such that $df_t/dt|_{t=0}=v\nu$ on $\partial\Omega$. 

Therefore, it follows as a consequence of \eqref{hadamardformula}, that $\Omega$ is an extremal domain if and only if its  first eigenfunctions of Laplacian operator with Dirichlet boundary condition solve the following overdetermined elliptic problem:
\begin{align*}\label{oep}
    \left\{
    \begin{array}{rl}
        \Delta\varphi + \lambda_1(\Omega)\varphi = 0 & \textrm{in } \Omega,\\
        \varphi = 0 & \textrm{on }  \partial\Omega,\\  
        \dfrac{\partial \varphi}{\partial\nu} = c & \textrm{on }  \partial\Omega,
    \end{array}
    \right.
\end{align*}
where \(c \neq 0\).

\subsection{Examples}\label{examples} 
Assume that there exists an isometric action of a compact Lie group \(\mathsf{G}\) on a Riemannian manifold \(M\), and let \(\Omega \subset M\) be a domain invariant under the action of \(\mathsf{G}\) with a single connected boundary component. If \(\varphi\) is a first eigenfunction, then \(\varphi\circ\theta\) also solves the Dirichlet eigenvalue problem in \(\Omega = \theta(\Omega)\) for all \(\theta \in \mathsf{G}\). Since \(\lambda_1(\Omega)\) is simple, it follows that \(\varphi\circ\theta = c_\theta \varphi\), defining a group homomorphism \(\theta \mapsto c_\theta\). Due to the compactness of \(\mathsf{G}\), we have \(c_\theta = 1\), proving that \(\varphi\) is \(\mathsf{G}\)-invariant. This invariance implies that \(\frac{\partial \varphi}{\partial \nu}\) is constant along \(\partial \Omega\).

A basic non-trivial example using this construction is given by the action \(\mathsf{G} = \mathsf{SO}(k+1) \times \mathsf{SO}(n-k)\) on the unit round sphere \(\mathbb{S}^{n} \subset \mathbb{R}^{n+1}\) when we consider \(\mathbb{R}^{n+1} = \mathbb{R}^{k+1} \oplus \mathbb{R}^{n-k}\). In this case, the invariant domains are the solid tori \(\Omega_r\), \(r \in (0, \pi/2)\), whose boundaries are given by the principal orbits of \(\mathsf{G}\), \(\Sigma(r) = \mathbb{S}^k(\cos r) \times \mathbb{S}^{n-k-1}(\sin r)\). In particular, $\partial \Omega(\pi/4)$ is a minimal hypersurface.
Note that, by choosing suitable values for \(r \in (0, \pi/2)\), we can also construct extremal \(\mathsf{G}\)-domains from this family with two boundary components.

Another family of examples of extremal domains we can construct using this method are the solid tori \(D^{n-1}_r \times \mathbb{S}^1 \subset T^n\), \(r \in (0,1)\), where \(T^n\) stands for the flat \(n\)-torus. From this particular family, Sicbaldi \cite{Sicbaldi10} and Schlenk-Sicbaldi \cite{SchlenkSicbaldi12} constructed new extremal domains using bifurcation techniques.

% The second variation formula

\section{Second variation formula for $\lambda_1$ and  stability of extremal domains}\label{SecondV}

Let $(M,g)$ and $\Omega_0\subset M$ be a smooth Riemannian manifold and a smooth compact domain, respectively. Suppose that $\Omega_0$ is extremal. Since $\Omega_0$ is a critical point of the functional $\Omega\subset M\mapsto  \lambda_1(\Omega)$, it is natural to investigate the second variation of this functional at $\Omega_0$. In this section, we will present a second variation formula for this functional, in a very general setting. As a consequence, we will give a criterion for a given extremal domain $\Omega_0\subset M$ to be stable.

\begin{definition}[Stable extremal domains]
We say that an extremal domain $\Omega_0\subset (M,g)$ is \textit{stable} if 
\begin{equation*}
\left.\dfrac{d^2}{dt^2}\right|_{t=0
}\lambda_1(\Omega_t) \geq  0,
\end{equation*}
for any volume preserving local deformation $\Omega_t=f_t(\Omega)$, $t\in(-\varepsilon,\epsilon)$, of $\Omega_0$ in $M$.
\end{definition}

In the following proposition, we present a formula for the second variation for the functional $\Omega\subset M\mapsto \lambda_1(\Omega)$ at an extremal domain $\Omega_0$, with respect to volume preserving local deformations.

\begin{proposition}[Second variation formula for $\lambda_1$]\label{2ndVF}
Let $\Omega_0\subset (M,g)$ be an extremal domain  and let $\varphi_0\in C^\infty(\Omega_0)$ be the first Dirichlet positive eigenfunction of the Laplacian operator on $\Omega_0$  with $\|\varphi_0\|_{L^2(\Omega_0)}=1$. For any volume preserving local deformation $\Omega_t=f_t(\Omega_0)$, $t\in(-\varepsilon,\varepsilon)$, of $\Omega_0$ in $M$, we have
\begin{equation}\label{2nd}
   \left.\dfrac{d^2}{dt^2}\right|_{t=0
}\lambda_1(\Omega_t) =2c^2\int_{\partial \Omega_0} \Big(v\frac{\partial \widehat{v}}{\partial \nu} +Hv^2\Big)\, d\ell,
\end{equation} 
where $H$ is the mean curvature of $\partial \Omega_0$ with respect to the outward unit normal vector $\nu$, $v=\left\langle df_t/dt|_{t=0},\nu\right\rangle$ is the normal displacement of $\partial\Omega_0$ induced by the deformation, $c=\partial\varphi_0/\partial\nu$ and $\widehat{v}$ is a $(\Delta+\lambda_1(\Omega_0))$-extension of $v$, that is, $\widehat{v}$ is a solution to the following problem:
\begin{equation*}\label{steklovextension}
	\left\{
	\begin{array}{rl}
		\Delta\widehat v+\lambda_1\widehat{v}=0& \textrm{in } \Omega_0,\\
		\widehat{v}=v & \textrm{on }  \partial\Omega_0.
	\end{array}
	\right.
\end{equation*}
\end{proposition}
\begin{proof}
See the Appendix \ref{appendix}.
\end{proof}

\begin{remark}\label{fredholmalternative}
Let $\Omega\subset (M,g)$ be a smooth compact domain and let $\varphi$ be a first Dirichlet eigenfunction  of the Laplacian operator on $\Omega$. Given $v\in C^\infty(\Sigma)$, as a consequence of the Fredholm alternative, there exists a $(\Delta+\lambda_1(\Omega))$-extension $\widehat{v}$ of $v$ if and only if
\begin{equation}\label{fredholmcondition}
\int_{\partial\Omega} v\dfrac{\partial \varphi}{\partial\nu}\,d\ell=0.
\end{equation}
Moreover, $\widehat{v}$ is unique up to an addition of a multiple of $\varphi$ (see, for example, Lemma 2.5 of \cite{tran1}). In the case $\Omega_0\subset (M,g)$ is an extremal domain, we know that $\partial \varphi_0/\partial\nu$ is constant along $\partial\Omega$. Thus, \eqref{fredholmcondition} becomes 
$$
\int_{\partial\Omega_0} v\,d\ell=0,
$$
which is equivalent to say that the local deformation of $\Omega_0$ given by $v$ is volume preserving.
\end{remark}

\begin{remark}
The second variation formula given by \eqref{2nd} was deduced by Garabedian and Schiffer \cite{GarabedianSchiffer2} for domains in the Euclidean plane $\mathbb{R}^2$ and by Shimakura \cite{Shimakura2} for domains in $\mathbb{R}^n$ for $n\geq 3$ (see also \cite[Theorem 2.5.6]{Henrot}). To the best of the authors' knowledge, this is the first time that this formula for extremal domains in general Riemannian manifolds appears in the literature. 
\end{remark}

We recall that volume preserving local deformations of \(\Omega_0\) are in bijection with functions in the space \(\mathcal{F}(\partial \Omega_0) = \{ v \in C^\infty(\partial \Omega_0) : \int_{\partial \Omega_0} v \, d\ell = 0 \}\) of smooth functions on the boundary of \(\Omega_0\) with zero average. Motivated by Proposition \ref{2nd}, we consider the quadratic form
\[
Q: \mathcal{F}(\partial \Omega_0) \times \mathcal{F}(\partial \Omega_0) \to \mathbb{R}
\]
given by
\[
Q(v,w) = \int_{\partial \Omega_0} \left( v \frac{\partial \widehat{w}}{\partial \nu} + Hvw \right) \, d\ell,
\]
where $\widehat{w}$ is a $(\Delta+\lambda_1(\Omega_0))$-extension of $w$ to $\Omega_0$. 
This setup allows us to define the index of an extremal domain.

\begin{definition}[Index of Extremal Domains]
If \(\Omega_0 \subset (M, g)\) is an extremal domain, then the Morse index of \(\Omega_0\) is given by
\begin{equation*}
\operatorname{Ind}(\Omega_0) = \max \{ \dim V : V \subset \mathcal{F}(\partial \Omega_0) \text{ and } Q \text{ is negative definite on } V \}.
\end{equation*}
\end{definition}

By this definition, the index represents the number of linearly independent, volume preserving variations that decrease the first Dirichlet eigenvalue. In particular, \(\Omega_0\) is stable if and only if \(\operatorname{Ind}(\Omega_0) = 0\).

\smallskip

Now, we have the following useful criterion for the stability of extremal domains.
\begin{proposition}\label{stabilitycriterion}
An extremal domain \(\Omega_0 \subset (M, g)\) is stable if and only if
\begin{equation*}\label{stabilityineq1}
\int_{\Omega_0} |\nabla v|^2 \, da - \lambda_1(\Omega_0) \int_{\Omega_0} v^2 \, da + \int_{\partial \Omega_0} H v^2 \, d\ell \geq 0,
\end{equation*}
for all \(v \in C^\infty(\Omega_0)\) such that \(\int_{\partial \Omega_0} v \, d\ell_g = 0\).
\end{proposition}
 
\begin{proof}
Given $v\in C^\infty(\Omega_0)$, define
$$
S(v,v)=S_0(v,v) + \int_{\partial\Omega_0} Hv^2\,d\ell,
$$
where $S_0(u,v)=\int_{\Omega_0} \langle \nabla u,\nabla v\rangle\,da-\lambda_1(\Omega_0)\int_{\Omega_0}uv\,da$. 

Integrating by parts, we obtain that 
\begin{equation}\label{eq1A}
S(\widehat{v|_{\partial\Omega_0}},\widehat{v|_{\partial\Omega_0}}) = Q(v|_{\partial\Omega_0},v|_{\partial\Omega_0}),
\end{equation}
where $\widehat{v|_{\partial\Omega_0}}$ is a $(\Delta+\lambda_1(\Omega_0))$-extension of $v|_{\partial\Omega_0}$.

Now, since $S_0(u,u)\geq 0$ for all $u\in C^\infty(\Omega_0)$ such that $u=0$ on $\partial\Omega_0$, we have that
\begin{equation}\label{eq1B}
S(\widehat{v|_{\partial\Omega_0}},\widehat{v|_{\partial\Omega_0}}) \leq S(v,v),
\end{equation}
for all $v\in C^\infty(\Omega_0)$. Therefore, from \eqref{eq1A} and \eqref{eq1B} we conclude the proof.
\end{proof}

We conclude this section with some properties of Jacobi functions associated to the first Dirichlet eigenvalue.

We say that a function \(v \in \mathcal F(\partial \Omega)\)  is a \textit{Jacobi function} of \(\lambda_1\) if
\begin{equation*}\label{jacobifunction}
\int_{\partial \Omega} \left( \dfrac{\partial \widehat{v}}{\partial \nu} + H v \right) w \, d\ell = 0,
\end{equation*}
for all \(w \in \mathcal F(\partial \Omega)\).

The next lemma presents basic facts about Jacobi functions of \(\lambda_1\) that will be used in Section \ref{mainthmproof}.

\begin{lemma}\label{basicfactsjacobi}
Let \(\Omega \subset (M,g)\) be an extremal domain.
\begin{itemize}
\item[(i)] \(v \in \mathcal F(\partial \Omega)\)  is a Jacobi function of \(\lambda_1\) if and only if there exist \(c \in \mathbb{R}\) and a \((\Delta + \lambda_1(\Omega))\)-extension \(\widehat{v}\) of \(v\) such that
\[
\dfrac{\partial \widehat{v}}{\partial \nu} + H v = c \quad \text{on } \partial \Omega.
\]
\item[(ii)] If \(\Omega\) is stable and \(v \in \mathcal F(\partial \Omega)\) satisfies
\[
\int_{\partial \Omega} \left( \dfrac{\partial \widehat{v}}{\partial \nu} + H v \right) v \, d\ell = 0,
\]
then \(v\) is a Jacobi function.
\end{itemize}
\end{lemma}

%Instability of annuli

\section{Instability of rotationally symmetric annuli in $\mathbb{S}^2$}

Given \(r_0 \in (0, \pi/2)\), let \(A_{r_0} \subset (\mathbb{S}^2, g_{can})\) be the rotationally symmetric domain given by \(A_{r_0} = \{ p \in \mathbb{S}^2 : d_{\mathbb{S}^2}(p, \gamma_0) \leq r_0 \}\), where \(\gamma_0 = \mathbb{S}^2 \cap \{ x_3 = 0 \}\).

Since \(A_{r_0}\) is rotationally symmetric and it is also symmetric with respect to \(\gamma_0\), it is not difficult to see that \(A_{r_0}\) is an extremal domain for all \(r_0 \in (0, \pi/2)\).

In this section, we will show that all these rotationally symmetric annuli are unstable extremal domains. We start with  the following lemma.
\begin{lemma}\label{lambda_1}
Let \(\Omega \subset \mathbb{S}^2\) be a stable extremal domain. If \(\partial \Omega\) has at least two connected components and \(\int_{\partial \Omega} x_i \, d\ell = 0\) for all \(i = 1, 2, 3\), then \(\lambda_1(\Omega) \leq 1\).
\end{lemma}
\begin{proof} By Proposition \ref{stabilitycriterion}, the stability of \(\Omega\) means that
\begin{equation}\label{ineqC}
\int_{\partial \Omega} \kappa_g v^2 \, d\ell + \int_{\Omega} |\nabla v|^2 \, da - \lambda_1(\Omega) \int_{\Omega} v^2 \, da \geq 0
\end{equation}
for all \(v \in C^{\infty}(\Omega)\) such that \(\int_{\partial \Omega} v \, d\ell = 0\), where $\kappa_g$ denotes the geodesic curvature of $\partial\Omega$.

Since \(\int_{\partial \Omega} x_i \, d\ell = 0\) for all \(i = 1, 2, 3\), we can use \(x_1, x_2\), and \(x_3\) as test functions in \eqref{ineqC}. By summing over \(i = 1, 2, 3\), we get
\begin{equation}\label{ineqD}
\int_{\partial \Omega} \kappa_g \, d\ell + 2 \operatorname{area}(\Omega) \geq \lambda_1(\Omega) \operatorname{area}(\Omega).
\end{equation}

Let \(r\) be the number of connected components of \(\partial \Omega\). Since \(r \geq 2\), we have by the Gauss-Bonnet Theorem that \(\int_{\partial \Omega} \kappa_g \, d\ell = -\operatorname{area}(\Omega) + 2\pi(2 - r) \leq -\operatorname{area}(\Omega)\). Thus, by this inequality and by \eqref{ineqD}, we have
\[
\operatorname{area}(\Omega) \geq \lambda_1(\Omega) \operatorname{area}(\Omega).
\]
Therefore, \(\lambda_1(\Omega) \leq 1\).
\end{proof}

Now we have:
\begin{proposition}
The rotationally symmetric annuli \(A_{r_0}\subset(\mathbb{S}^2, g_{can})\) are not stable for any \(r_0 \in (0, \pi/2)\).
\end{proposition}
\begin{proof}
First, note that we can write \(g_{can}\) on \(A_{r_0}\) as \(g_{can} = dr^2 + \cos^2 r \, d\theta^2\) on \([-r_0, r_0] \times \mathbb{S}^1\), where \(r_0 \in (0, \pi/2)\).

We note that the first eigenfunction of \(\Delta_{g_{can}}\) on \(A_{r_0}\) with Dirichlet boundary condition is a radial function, that is, it depends only on \(r\).

If \(\varphi = \varphi(r)\), then
\[
\Delta_{g_{can}} \varphi = \varphi'' - (\tan r) \varphi'.
\]

We would like to investigate the first eigenvalue of the following problem:
\begin{equation*}\label{eqA}
\varphi'' - (\tan r) \varphi' + \lambda \varphi = 0, \quad \varphi(-r_0) = \varphi(r_0) = 0.
\end{equation*}

Define \(f(r) = \sqrt{\cos r}\). Since \(f\) solves
\begin{equation}\label{eqB}
f' = -\frac{1}{2} (\tan r) f,
\end{equation}
we have, by a simple computation, that \(w = f \varphi\) solves
\begin{equation*}\label{eqC}
w'' - \frac{f''}{f} w + \lambda w = 0, \quad w(-r_0) = w(r_0) = 0.
\end{equation*}

As a consequence, we have
\begin{equation}\label{eqD}
\lambda_1(A_{r_0}) = \inf \frac{\int_{-r_0}^{r_0} (w')^2 + \left( \frac{f''}{f} \right) w^2 \, dr}{\int_{-r_0}^{r_0} w^2 \, dr},
\end{equation}
where the \(\inf\) is taken over the set of functions \(w\) defined on \([-r_0, r_0]\) such that \(w(-r_0) = w(r_0) = 0\).

Our next goal is to find a lower estimate for \(\lambda_1(A_{r_0})\).

First, it follows from \eqref{eqB} that
\begin{equation*}\label{eqE}
\frac{f''}{f} = -\frac{1}{2} - \frac{1}{4} \tan^2 r.
\end{equation*}

Moreover, we have
\begin{equation*}
\frac{\int_{-r_0}^{r_0} (w')^2 \, dr}{\int_{-r_0}^{r_0} w^2 \, dr} \geq \left( \frac{\pi}{2r_0} \right)^2.
\end{equation*}

Note that for \(|r| \leq r_0 < \pi/3\), we have
\begin{equation}\label{ineqA}
\left( \frac{f''}{f} \right)(r) \geq \left( \frac{f''}{f} \right)(r_0) > -\frac{1}{2} - \frac{1}{4} \tan^2(\pi/3) = -\frac{5}{4}
\end{equation}
and
\begin{equation}\label{ineqB}
\left( \frac{\pi}{2r_0} \right)^2 > \frac{9}{4}
\end{equation}
for \(|r| \leq r_0 < \pi/3\).

Thus, it follows from \eqref{eqD} and estimates \eqref{ineqA} and \eqref{ineqB} that
\[
\lambda_1 \geq \left( \frac{\pi}{2r_0} \right)^2 + \left( \frac{f''}{f} \right)(r_0) > 1.
\]
From this and by using Lemma \ref{lambda_1}, we can conclude the instability of the extremal annuli \(A_{r_0} \subset \mathbb{S}^2\) for all \(r_0 < \pi/3\).

\smallskip

Now, suppose that \(r_0 \geq \pi/3\) and define \(\varphi(r) = \sin r / \sin r_0\) on \(A_{r_0}\). Note that \(\varphi(r_0) = 1\) and \(\varphi(-r_0) = -1\). Thus, \(\int_{\partial A_{r_0}} \varphi \, d\ell = 0\).

We have
\begin{eqnarray*}
\int_{\partial A_{r_0}} \kappa_g \varphi^2 \, d\ell + \int_{A_{r_0}} |\nabla \varphi|^2 \, da - \lambda_1(A_{r_0}) \int_{A_{r_0}} \varphi^2 \, da 
&<& \int_{\partial A_{r_0}} \kappa_g + \int_{A_{r_0}} |\nabla \varphi|^2 \, da \\
&=& -\operatorname{area}(A_{r_0}) + \int_{A_{r_0}} |\nabla \varphi|^2 \, da,
\end{eqnarray*}
where we have used the Gauss-Bonnet Theorem in the last equality.

We have
\begin{equation}\label{eqF}
\operatorname{area}(A_{r_0}) = 2\pi \int_{-r_0}^{r_0} \cos r \, dr = 4\pi \sin r_0
\end{equation}

\begin{equation}\label{eqG}
\int_{A_{r_0}} |\nabla \varphi|^2 \, da = \frac{2\pi}{\sin^2 r_0} \int_{-r_0}^{r_0} \cos^3 r \, dr = \frac{4\pi}{\sin r_0} - \frac{4\pi}{3} \sin r_0.
\end{equation}

By \eqref{eqF} and \eqref{eqG} we have
\[
\int_{\partial A_{r_0}} \kappa_g \varphi^2 \, d\ell + \int_{A_{r_0}} |\nabla \varphi|^2 \, da - \lambda_1(A_{r_0}) \int_{A_{r_0}} \varphi^2 \, da < \frac{4\pi}{\sin r_0} \left(1 - \frac{4}{3} \sin^2 r_0 \right) \leq 0,
\]
since \(r_0 \geq \pi/3\).

Thus, we conclude that the annulus \(A_{r_0} \subset \mathbb{S}^2\) is also unstable for \(r_0 \geq \pi/3\).

\end{proof}

\section{Proof of Theorem \ref{maintheorem}}\label{mainthmproof}

Let $\Omega\subset\mathbb{S}^2$ be a stable extremal domain and let $\partial^1\Omega,\partial^2\Omega,\ldots,\partial^r\Omega$, $r\geq 1$, be the connected components of $\partial\Omega$.

For each unit vector $a\in\mathbb{R}^3$, define the following vector field on $\mathbb{S}^2$: $V_a(x)=a\wedge x$, where $\wedge$ denotes the vector product in $\mathbb{R}^3$. 

Now, define $\varphi_a:\partial\Omega\to\mathbb{R}$ by $\varphi_a(x)=\langle V_a(x),\nu(x)\rangle $, where $\nu$ denotes the outward unit normal vector along $\partial\Omega$. Since $\dive_{\mathbb{S}^2} V_a=0$, we have 
\begin{equation}\label{eqH}
\int_{\partial^k\Omega}\varphi_a\,d\ell=0,
\end{equation}
for all $k=1,2,\ldots,r$.

The vector field $V_a$ generates a one-parameter family of rotations $f_t:\mathbb{S}^2\to\mathbb{S}^2$ around the vector $a$. Consider the local deformation of $\Omega$ in $\mathbb{S}^2$ given by $\Omega_t=f_t(\Omega)$. As $f_t$ is an isometry of $\mathbb{S}^2$ for all $t$, we have that $\lambda_1(\Omega_t)=\lambda_1(\Omega)$ for all $t$. In particular, 
$$
\left.\dfrac{d^2}{dt^2}\lambda_1(\Omega_t)\right|_{t=0}=0
$$
and, by Proposition \ref{2ndVF}, this implies that
$$
\int_{\partial\Omega}\left(\dfrac{\partial\widehat{\varphi_a}}{\partial\nu}+\kappa_g\varphi_a\right)\varphi_a\,d\ell=0.
$$

Therefore, since $\Omega$ is stable, we have by Lemma \ref{basicfactsjacobi} that $\varphi_a$ is a Jacobi function of $\lambda_1$, that is, there exists a function $\widehat{\varphi_a}\in C^\infty(\Omega)$ such that 
\begin{equation}\label{eqL}
\left\{
\begin{array}{rcl}
\Delta\widehat{\varphi_a}+\lambda_1(\Omega)\widehat{\varphi_a}=0 & \mbox{on} & \Omega,\vspace{0.2cm}\\
\dfrac{\partial \widehat{\varphi_a}}{\partial\nu}+\kappa_g\varphi_a=c & \mbox{on} & \partial\Omega.
\end{array}
\right.
\end{equation}

After adding $\widehat{\varphi_a}$ to a constant multiple of $\varphi$, if necessary, where $\varphi$ is the first eigenfuncion of the Laplacian operator on $\Omega$ with Dirichlet boundary condition such that $\varphi>0$ and  $\|\varphi\|_{L^2(\Omega)}=1$, we may assume without loss of generality that $c=0$ in the second equation above.

Next, fix \(k \in \{1,2,\ldots,r\}\). We know that \(\mathbb{S}^2 \setminus \partial^k \Omega\) consists of two connected components, both homeomorphic to a disk. Let \(W\) be one of these connected components. Note that \(\partial W = \partial^k \Omega\). Now, let \(D\) be the closed geodesic disk of largest radius such that \(D \subset \overline{W}\). Because of the way \(D\) was chosen, we have that \(\partial D\) touches \(\partial^k \Omega\) tangentially at least at two points. In particular, if \(a \in D\) is the center of \(D\), then \(\varphi_a\) vanishes at least at two points on \(\partial^k \Omega\). We claim that \(\varphi_a\) vanishes at least at one more point of \(\partial^k \Omega\). In fact, if \(-a \in \partial^k \Omega\), we are done. If not, we let \(W'\) be the connected component of \(\mathbb{S}^2 \setminus \partial^k \Omega\) which contains \(-a\) and consider \(D'\), the geodesic closed disk centered at \(-a\) contained in \(\overline{W'}\) with the largest possible radius. In this case, we have that \(\partial D'\) touches \(\partial^k \Omega\) tangentially at least at one point. Thus, this proves our claim.

Following ideas from Ros and Vergasta \cite{RV}, we will prove in the following that $\Omega\setminus \widehat{\varphi_a}^{-1}(0)$ has at least three connected components. Let $M_1,\ldots,M_l$ be the connected components of $\Omega\setminus \widehat{\varphi_a}^{-1}(0)$. By applying the Gauss-Bonnet Theorem to each $M_i$ we obtain
\begin{equation*}\label{eqI}
\area(M_i)=2\pi\chi(M_i)-\int_{\partial M_i}\kappa_g\,d\ell-\sum_{k=1}^{j_i}\theta^i_{k},
\end{equation*}
where  $\theta^i_k$ are the external angles of $\partial M_i$.

By summing over $i=1,2,\ldots,l$, we get that
\begin{equation*}\label{eqJ}
\area(\Omega)=2\pi\left(\sum_{i=1}^l\chi(M_i)\right)-\int_{\partial\Omega}\kappa_g\,d\ell-\sum_j\theta_j,
\end{equation*}
where $\sum_j\theta_j$ stands for the sum of the external angles of all nodal domains $M_i$.

By the Gauss-Bonnet Theorem applied to $\Omega$, we have
\begin{equation}\label{eqK}
2\pi(2-r)=2\pi\left(\sum_{i=1}^l\chi(M_i)\right)-\sum_j\theta_j.
\end{equation}

Now, it follows from \eqref{eqH} that $\varphi_a$ vanishes in at least two points of $\partial^i\Omega$, for all $i=1,2,\ldots,r$. Moreover, we know that for $\varphi_a$ vanishes in at least three points of $\partial^k\Omega$. This implies that
\begin{equation}\label{ineqE}
\sum_j\theta_j\geq 2\pi\left(\dfrac{1}{2}+r\right).
\end{equation}

Thus, by \eqref{eqK} and \eqref{ineqE} we can conclude that
$$
\sum_{i=1}\chi(M_i)\geq \dfrac{5}{2}>2.
$$
This implies that $\Omega\setminus (\widehat{\varphi_a})^{-1}(0)$ has at least three connected components.

But since $\Omega$ is a stable extremal domain and $\widehat{\varphi_a}$ solves \eqref{eqL} we have that $\Omega\setminus(\widehat{\varphi_a})^{-1}(0)$ has at most two connected components unless $\widehat{\varphi_a}\equiv 0$. Thus, it follows from the above argument that $\widehat{\varphi_a}\equiv 0$, which implies that $\Omega$ is rotationally symmetric.

Since the rotationally symmetric extremal annuli are unstable we have that $\Omega$ is a geodesic ball.

\section{Proof of Theorem \ref{topbound}}

Let $(M^2,g)$ be an orientable Riemannian surface and let $\Omega\subset M$ be a stable extremal domain. Suppose that $\Omega$ has nonnegative total Gauss curvature, that is, $\int_\Omega K_g\,da\geq 0$.

By a result of Gabard \cite[Théorème 7.2.]{Gab}, which improved a previous result due to Alfhors \cite{Ahl},  there exists a proper conformal branched cover $\varphi=(\varphi_1,\varphi_2):\Omega\to\mathbb{D}^2$, where $\mathbb{D}^2\subset\mathbb{R}^2$ is the closed unit disk, of degree at most $g_\Omega+r$. Moreover, after composing $\varphi$ with a conformal diffeomorphism of $\mathbb{D}^2$ if necessary, we may suppose that
$$
\int_{\partial\Omega}\varphi_i\,d\ell=0,
$$
for  $i=1,2$.

Thus, by the stability of $\Omega$, we have that
$$
\sum_{i=1}^2\int_{\partial\Omega}\kappa_g \varphi_i^2\,d\ell+\sum_{i=1}^2\int_{\Omega}|\nabla\varphi_i|^2\,da\geq\lambda_1(\Omega)\sum_{i=1}^2\int_\Omega\varphi_i^2\,da .
$$

Since $\varphi$ is conformal and has degree at most $g_\Omega+r$, 
$$
\sum_{i=1}^2\int_\Omega|\nabla \varphi_i|^2\,da\leq 2\pi(g_\Omega+r).
$$
Thus, it follows from this and from the Gauss-Bonnet Theorem that
$$
2\pi(2-2g_\Omega-r)+2\pi(g_\Omega+r)\geq \lambda_1(\Omega) \sum_{i=1}^2\int_\Omega\varphi_i^2\,da+\int_\Omega K_g\,da> 0.
$$
Thus, $4\pi-2\pi g_\Omega>0$, that is, $g_\Omega<2$. 

Now, in order to estimate the number $r$ of connected components of $\partial\Omega$, we use a similar, but different, balancing argument. 
Let $\overline{\Omega}$ denote a compact Riemannian surface obtained from $\Omega$ by attaching a conformal disk at any connected component of $\partial \Omega$, such that 
$\overline{\Omega}$ and $\Omega$ have the same genus.
There exists a nonconstant holomorphic map $\psi=(\psi_1,\psi_2,\psi_3):\overline{\Omega}\to\mathbb{S}^2$ of degree at most $1+\left\lfloor (g_\Omega+1)/2\right\rfloor$ (see, for example, \cite{GH}, p.261), where $\mathbb{S}^2\subset\mathbb{R}^3$ is the unit sphere centered at the origin.
Again, after composing $\psi$ with a conformal diffeomorphism of $\mathbb{S}^2$ if necessary, we may assume that
$$
\int_{\partial\Omega}\psi_i\,d\ell=0,
$$
for $i=1,2,3$.

For each $i=1,2,3$, we define 
$\varphi_i=\psi_i|_{\Omega}$. Since $\sum_{i=1}^3\varphi_i^2=1$ in $\Omega$, it follows from the stability of $\Omega$ that
\begin{align*}
\int_{\partial\Omega}\kappa_g\,d\ell+8\pi\left(1+\left\lfloor\dfrac{g_\Omega+1}{2}\right\rfloor\right)\geq&\int_{\partial\Omega}\kappa_g\,d\ell+\sum_{i=1}^3\int_{\overline{\Omega}}|\nabla \psi_i|^2\,da\\
>&\int_{\partial\Omega}\kappa_g\, d\ell +\sum_{i=1}^3 \int_{\Omega} |\nabla\varphi_i|^2\,da\\
\geq&\, \lambda_1(\Omega)\area(\Omega)>0.
\end{align*}

By the Gauss-Bonnet Theorem, we have that
$$
12\pi-2\pi r -4\pi g_\Omega +8\pi \left\lfloor\dfrac{g_\Omega+1}{2}\right\rfloor>0.
$$

Therefore, if $g_\Omega=0$, then $12\pi-2\pi r>0$ which implies that $r< 6$. In the case $g_\Omega=1$, we have $16\pi -2\pi r>0$ and this implies $r< 8$. 

\section{Proof of Theorem \ref{topbound2}}

Let $(M^2,g)$ be an orientable Riemannian surface. First, note that it follows from the asymptotic expansion of the Faber-Krahn profile proved by Druet in \cite{Druet} that
\begin{equation}\label{eq1top2}
\lim_{m\to 0}FK(m)m=\lambda_1(\mathbb{D}^2)\pi,
\end{equation}
where $\mathbb{D}^2\subset \mathbb{R}^2$ is the closed unit disk.

It is well known that $\lambda_1(\mathbb{D}^2)=j_0^2$, where $j_0$ is the first positive zero of the Bessel function $J_0$ (see, for example, Theorem 4, Section 5, Chapter II of \cite{Chavel84}). Moreover, since $j_0^2\approx 2.40483$ we have that $\lambda_1(\mathbb{D}^2)\approx 5.7832>5$.  Therefore, by \eqref{eq1top2}, we conclude that there exists $\varepsilon>0$ such that 
\begin{equation}\label{eq2top2}
FK(m)m>5\pi
\end{equation}
for all $0<m<\varepsilon$.

Now, suppose that $\Omega\subset M$ is a stable extremal domain. Following the same steps as the proof of Theorem \ref{topbound} we can obtain the following two inequalities:
\begin{equation*}
4\pi- 2\pi g_\Omega>\int_\Omega K_g\,da\geq -\sup_M|K_g|\area(\Omega)
\end{equation*}
and
\begin{eqnarray*}\label{eq3top2}
12\pi-2\pi r-4\pi g_\Omega+8\pi\left\lfloor\dfrac{g_\Omega+1}{2}\right\rfloor&\geq& \lambda_1(\Omega)\area(\Omega)+\int_\Omega K_g\,da\\
&\geq& \lambda_1(\Omega)\area(\Omega)-\sup_M|K_g|\area(\Omega).
\end{eqnarray*}

Therefore, decreasing $\varepsilon>0$ if necessary, we obtain from the first inequality that if $\area(\Omega)<\varepsilon$ then $g_\Omega\leq 2$. Moreover, since $\lambda_1(\Omega)\area(\Omega)>5\pi$ by \eqref{eq2top2}, 
we get from the second inequality that
$$
7\pi - 2\pi g_\Omega >-\sup_M|K_g|\area(\Omega)\ \ \mbox{if $g_\Omega\in\{0,2\}$}
$$
and
$$
11\pi - 2\pi r >-\sup_M|K_g|\area(\Omega) \ \ \mbox{if $g_\Omega=1$.}
$$
Thus, decreasing $\varepsilon>0$ once more if necessary, we obtain that $r\leq 3$ if $g_\Omega\in \{0,2\}$ and $r\leq 5$ if $g_\Omega=1$ whenever $\area(\Omega)<\varepsilon$.
\section{Proof of Theorem \ref{theoremhigherdim}}

Suppose that \(\Omega \subset (\mathbb{S}^n, g_{can})\) is a stable extremal domain such that \(\partial \Omega\) is a compact minimal hypersurface, that is, \(H_{\partial \Omega} = 0\). It follows from a result due to Reilly \cite[Theorem 4]{Reilly77} that \(\lambda_1(\Omega) \geq n\) with equality if and only if \(\Omega\) is a hemisphere.

Since \(\partial \Omega\) is minimal, it is well known that
\[
\int_{\partial \Omega} x_i \, d\ell = 0,
\]
for all \(i = 1, 2, \ldots, n+1\).

Thus, using the coordinate functions as test functions in Proposition \ref{stabilitycriterion}, we get
\begin{equation*}\label{ineq1thm3}
\int_{\Omega} |\nabla x_i|^2 \, da \geq \lambda_1(\Omega) \int_{\Omega} x_i^2 \, da
\end{equation*}
for all \(i = 1, 2, \ldots, n+1\).

By summing over \(i = 1, 2, \ldots, n+1\) and using that \(|\nabla x_i|^2 = 1 - x_i^2\), we have
\[
n \operatorname{vol}(\Omega) \geq \lambda_1(\Omega) \operatorname{vol}(\Omega).
\]
This implies \(\lambda_1(\Omega) \leq n\). Thus, \(\lambda_1(\Omega) = n\) and \(\Omega\) is a hemisphere.

% Appendix

\appendix 
\section{Proof of Proposition \ref{2ndVF}}\label{appendix}

Let $(M,g)$ be a Riemannian manifold and let $\Omega_0 \subset M$ be an extremal domain. In this section, we will deduce the second variation formula \eqref{2nd} for the first eigenvalue functional $\Omega \mapsto \lambda_1(\Omega)$ at $\Omega_0$.

Let $\varphi_0 \in C^\infty(\Omega_0)$ be the first eigenfunction of the Laplacian operator on $\Omega_0$ with Dirichlet boundary condition such that $\varphi_0 > 0$ on $\Omega_0$ and $\|\varphi_0\|_{L^2(\Omega_0)} = 1$. Since $\Omega_0$ is extremal, there is a constant $c \in \mathbb{R}$ such that $\partial \varphi_0 / \partial \nu_0 = c$ on $\partial \Omega_0$.

For each smooth domain $\Omega \subset M$, define
\begin{equation*}
J(\Omega) = \lambda_1(\Omega) + c^2 \vol(\Omega).
\end{equation*}

Let $\Omega_t = f_t(\Omega_0)$, $t \in (-\varepsilon, \varepsilon)$, be a local deformation of $\Omega_0$ in $M$ given by a smooth vector field $V \in \mathfrak{X}(M)$. It is well known that
\[
\dfrac{d}{dt}\vol(\Omega_t) = \int_{\partial \Omega_t} v_t \, d\ell_t, \quad \forall\, t \in (-\varepsilon, \varepsilon),
\]
where $v_t = \left\langle V, \nu_t \right\rangle$, $\nu_t$ denotes the outward unit normal vector along $\partial \Omega_t$, and $d\ell_t$ is the volume element of $\partial \Omega_t$ induced by $g$.

Therefore, by Proposition \ref{1stvarformula}, we have
\begin{equation}\label{J1stvariation}
\dfrac{d}{dt} J(\Omega_t) = -\int_{\partial \Omega_t} v_t \left[\left(\dfrac{\partial \varphi_t}{\partial \nu_t}\right)^2 - c^2 \right] \, d\ell_t, \quad \forall\, t \in (-\varepsilon, \varepsilon),
\end{equation}
where $\varphi_t \in C^\infty(\Omega_t)$ denotes the positive first eigenfunction of the Laplacian operator on $\Omega_t$ with Dirichlet boundary condition and such that $\|\varphi_t\|_{L^2(\Omega_t)} = 1$.

Since $\partial \varphi_0 / \partial \nu_0 = c$, it follows from \eqref{J1stvariation} that $\Omega_0$ is a critical point of $J$ with respect to all local deformations $\Omega_t = f_t(\Omega_0)$ of $\Omega_0$ in $M$, not only the volume preserving ones. Moreover, we have

\begin{align*}
    \left. \dfrac{d^2}{dt^2} J(\Omega_t) \right|_{t=0} &= -\int_{\partial \Omega_0} v_0 \left[ \left. \dfrac{d}{dt} \left( \frac{\partial \varphi_t}{\partial \nu_t} \right)^2 \right|_{t=0} \right] \, d\ell_0 \\
    &= -2 \int_{\partial \Omega_0} v_0 \dfrac{\partial \varphi_0}{\partial \nu_0} \left( \left. \dfrac{d}{dt} \dfrac{\partial \varphi_t}{\partial \nu_t} \right|_{t=0} \right) \, d\ell_0.
\end{align*}

From now on, we will denote derivatives with respect to $t$ at $t=0$ by using dots.

Note that
\begin{align*}
\dot{\left( \frac{\partial \varphi_t}{\partial \nu_t} \right)}
&= \frac{d}{dt} \Big|_{t=0} \langle \nabla \varphi_t, \nu_t \rangle \\
&= \langle \nabla \dot{\varphi}, \nu_0 \rangle + \langle \nabla \varphi_0, \dot{\nu} \rangle \\
&= \frac{\partial \dot{\varphi}}{\partial \nu_0} + \frac{\partial \varphi_0}{\partial \dot{\nu}}.
\end{align*}

Let $g_t = f_t^*(g)$ be the pullback metric of $g$ on $\Omega_0$. Let $h = \dot{g}$ be the derivative tensor at $t=0$. Note that $h = \mathcal{L}_V g$, where $\mathcal{L}$ denotes the Lie derivative.

Since $\nabla \varphi_0 = \left( \partial \varphi_0 / \partial \nu_0 \right) \nu_0$ on $\partial \Omega_0$, we get $\langle \nabla \varphi_0, \dot{\nu} \rangle = \left( \partial \varphi_0 / \partial \nu_0 \right) \langle \nu_0, \dot{\nu} \rangle$. Moreover, since $\langle \nu_t, \nu_t \rangle = 1$ for all $t$, it follows directly that
\[
\langle \nu_0, \dot{\nu} \rangle = -\dfrac{1}{2} h(\nu_0, \nu_0).
\]

Therefore,
\begin{align*}
    \left. \dfrac{d^2}{dt^2} J(\Omega_t) \right|_{t=0} &= -2 \int_{\partial \Omega_0} v_0 \dfrac{\partial \varphi_0}{\partial \nu_0} \dfrac{\partial \dot{\varphi}}{\partial \nu_0} \, d\ell_0 + \int_{\partial \Omega_0} v_0 \left( \dfrac{\partial \varphi_0}{\partial \nu_0} \right)^2 h(\nu_0, \nu_0) \, d\ell_0 \\
    &= -2c \int_{\partial \Omega_0} v_0 \dfrac{\partial \dot{\varphi}}{\partial \nu_0} \, d\ell_0 + c^2 \int_{\partial \Omega_0} v_0 h(\nu_0, \nu_0) \, d\ell_0.
\end{align*}

Now, suppose that $\Omega_t = f_t(\Omega_0)$, $t \in (-\varepsilon, \varepsilon)$, is a volume preserving local deformation. Note that, in this case, we have
\[
\left. \dfrac{d^2}{dt^2} J(\Omega_t) \right|_{t=0} = \left. \dfrac{d^2}{dt^2} \lambda_1(\Omega_t) \right|_{t=0}.
\]
Moreover, we also have that $\int_{\partial \Omega_0} v_0 \, d\ell_0 = 0$.

Let $\widehat{v}_0$ be a $(\Delta + \lambda_1(\Omega_0))$-extension of $v_0$, that is, $\widehat{v}_0$ solves the following problem (see Remark \ref{fredholmalternative}):
\begin{equation*}
    \left\{
    \begin{array}{rl}
        \Delta \widehat{v}_0 + \lambda_1(\Omega_0) \widehat{v}_0 = 0 & \textrm{in } \Omega_0, \\
        \widehat{v}_0 = v_0 & \textrm{on } \partial \Omega_0.
    \end{array}
    \right.
\end{equation*}

Using that $\dot{\varphi} = 0$ along $\partial \Omega_0$, the Green formula yields
\begin{align*}
\int_{\partial \Omega_0} v_0 \frac{\partial \dot{\varphi}}{\partial \nu_0} \, d\ell_0
&= \int_{\Omega_0} \left( \widehat{v}_0 \Delta \dot{\varphi} - \dot{\varphi} \Delta \widehat{v}_0 \right) \, da_0 \\
&= \int_{\Omega_0} \widehat{v}_0 \left( \Delta \dot{\varphi} + \lambda_1(\Omega_0) \dot{\varphi} \right) \, da_0,
\end{align*}
where $da_0$ is the volume element of $\Omega_0$.

Now, differentiating the equation $\Delta_{g_t} \varphi_t + \lambda_1(\Omega_t) \varphi_t = 0$ we get
\[
\dot{\Delta} \varphi_0 + \Delta \dot{\varphi} + \dot{\lambda}_1 \varphi_0 + \lambda_1(\Omega_0) \dot{\varphi} = 0.
\]
And so, since $\dot{\lambda}_1 = 0$ we have
\[
\Delta \dot{\varphi} + \lambda_1(\Omega_0) \dot{\varphi} = -\dot{\Delta} \varphi_0.
\]

Combining these identities, we arrive at
\begin{equation}\label{eqsegundaderivalambda}
   \left. \dfrac{d^2}{dt^2} \lambda_1(\Omega_t) \right|_{t=0} =
2c \int_{\Omega_0} \widehat{v}_0 \dot{\Delta} \varphi_0 \, da_0 + c^2 \int_{\partial \Omega_0} v_0 h(\nu_0, \nu_0) \, d\ell_0,
\end{equation}
where $\dot{\Delta} \varphi_0 = -\langle h, \nabla^2 \varphi_0 \rangle - \dive h(\nabla \varphi_0) + \frac{1}{2} \langle \nabla (\tr_{\Omega_0} h), \nabla \varphi_0 \rangle$ (see \cite{Berger73, ElsoufiIlias07}).

Integrating by parts, we have
\begin{align*}
-\int_{\Omega_0} \langle h, \nabla^2 \varphi_0 \rangle \widehat{v}_0 \, da_0 &= -\int_{\Omega_0} \dive (h(\nabla \varphi_0, \cdot) \widehat{v}_0) \, da_0 + \int_{\Omega_0} \dive h(\nabla \varphi_0) \widehat{v}_0 \, da_0 \\
&\quad + \int_{\Omega_0} h(\nabla \varphi_0, \nabla \widehat{v}_0) \, da_0 \\
&= -\int_{\partial \Omega_0} h(\nabla \varphi_0, \nu_0) \widehat{v}_0 \, d\ell_0 + \int_{\Omega_0} \dive h(\nabla \varphi_0) \widehat{v}_0 \, da_0 \\
&\quad + \int_{\Omega_0} h(\nabla \varphi_0, \nabla \widehat{v}_0) \, da_0.
\end{align*}

Plugging this identity into \eqref{eqsegundaderivalambda} and noting that
\[
h(\nabla \varphi_0, \nu_0) = \left( \partial \varphi_0 / \partial \nu_0 \right) h(\nu_0, \nu_0) = c h(\nu_0, \nu_0) \quad \text{on } \partial \Omega_0,
\]
we get
\begin{align*}
\left. \dfrac{d^2}{dt^2} \lambda_1(\Omega_t) \right|_{t=0} &= -2c \int_{\partial \Omega_0} h(\nabla \varphi_0, \nu_0) v_0 \, d\ell_0 + 2c \int_{\Omega_0} h(\nabla \varphi_0, \nabla \widehat{v}_0) \, da_0 \\
&\quad + c \int_{\Omega_0} \widehat{v}_0 \langle \nabla (\tr_{\Omega_0} h), \nabla \varphi_0 \rangle \, da_0 + c^2 \int_{\partial \Omega_0} v_0 h(\nu_0, \nu_0) \, d\ell_0 \\
&= -c^2 \int_{\partial \Omega_0} v h(\nu_0, \nu_0) \, d\ell_0 + 2c \int_{\Omega_0} h(\nabla \varphi_0, \nabla \widehat{v}_0) \, da_0 \\
&\quad + c \int_{\Omega_0} \langle \nabla (\tr_{\Omega_0} h), \nabla \varphi_0 \rangle \widehat{v}_0 \, da_0.
\end{align*}

Using $h = \mathcal{L}_V g$, we have
\begin{align*}
\int_{\Omega_0} h(\nabla \varphi_0, \nabla \widehat{v}_0) \, da_0
&= \int_{\Omega_0} \left( \mathcal{L}_V g \right)(\nabla \varphi_0, \nabla \widehat{v}_0) \, da_0 \\
&= \int_{\Omega_0} \left( \langle \nabla_{\nabla \varphi_0} V, \nabla \widehat{v}_0 \rangle + \langle \nabla_{\nabla \widehat{v}_0} V, \nabla \varphi_0 \rangle \right) \, da_0.
\end{align*}

On the other hand,
\begin{align*}
\text{div}(\langle V, \nabla \widehat{v}_0 \rangle \nabla \varphi_0)
&= \langle \nabla \langle V, \nabla \widehat{v}_0 \rangle, \nabla \varphi_0 \rangle + \langle V, \nabla \widehat{v}_0 \rangle \Delta \varphi_0 \\
&= \langle \nabla_{\nabla \varphi_0} V, \nabla \widehat{v}_0 \rangle + \langle V, \nabla_{\nabla \varphi_0} \nabla \widehat{v}_0 \rangle - \lambda_1(\Omega_0) \langle V, \nabla \widehat{v}_0 \rangle \varphi_0 \\
&= \langle \nabla_{\nabla \varphi_0} V, \nabla \widehat{v}_0 \rangle + \nabla^2 \widehat{v}_0(V, \nabla \varphi_0) - \lambda_1(\Omega_0) \langle V, \nabla \widehat{v}_0 \rangle \varphi_0,
\end{align*}
and similarly,
\begin{align*}
\text{div}(\langle V, \nabla \varphi_0 \rangle \nabla \widehat{v}_0)
&= \langle \nabla_{\nabla \widehat{v}_0} V, \nabla \varphi_0 \rangle + \nabla^2 \varphi_0(V, \nabla \widehat{v}_0) - \lambda_1(\Omega_0) \langle V, \nabla \varphi_0 \rangle \widehat{v}_0.
\end{align*}

Thus, we have that
\begin{align*}
\int_{\Omega_0} h(\nabla \varphi_0, \nabla \widehat{v}_0)\,da_0
=&
\int_{\Omega_0} \text{div}(\langle V, \nabla \widehat{v}_0\rangle \nabla \varphi_0)\,da_0
-\int_{\Omega_0}  \nabla^2 \widehat{v}_0(V,\nabla \varphi_0)\,  da_0\\ 
&+ \lambda_1(\Omega_0)\int_{\Omega_0}\langle V,\widehat{v}_0\rangle\varphi_0\,da_0 +\int_{\Omega_0}  \dive(\langle V, \nabla \varphi_0\rangle \nabla \widehat{v}_0)\,da_0\\
&-\int_{\Omega_0}  \nabla^2 \varphi_0(V,\nabla \widehat{v}_0)\,   da_0+\lambda_1(\Omega_0)\int_{\Omega_0} \langle V, \nabla \varphi_0\rangle\widehat{v}_0\, da_0\\
=&\int_{\partial \Omega_0}\left(\langle V, \nabla \widehat{v}_0\rangle\frac{\partial \varphi_0}{\partial \nu_0}+\langle V, \nabla \varphi_0\rangle \frac{\partial \widehat{v}_0}{\partial \nu_0}
\right)\,d\ell_0\\
& - \int_{\Omega_0} \left( \nabla^2 \widehat{v}_0(V,\nabla \varphi_0)+\nabla^2 \varphi_0(V,\nabla \widehat{v}_0) \right)\,  da_0\\
&+\lambda_1(\Omega_0)\int_{\Omega_0}  \langle V, \nabla \widehat{v}_0\rangle\varphi_0\,  da_0+\lambda_1(\Omega_0)\int_{\Omega_0} \langle V, \nabla \varphi_0\rangle\widehat{v}_0\, da_0.
\end{align*}

Now, on $\partial\Omega_0$,  we have $V=V^T+v_0\nu_0$ and $\nabla \varphi_0=c\nu_0$, where $V^T$ is the component of $V$ tangent to $\partial\Omega_0$. Moreover, note that 
$$
\nabla^2 \widehat{v}_0(V,\nabla \varphi_0)+\nabla^2 \varphi_0(V,\nabla \widehat{v}_0) = \langle\nabla \langle\nabla \widehat{v}_0,\nabla\varphi_0\rangle,V\rangle.
$$
Thus

\begin{align*}
\int_{\Omega_0} h(\nabla \varphi_0, \nabla \widehat{v}_0 )\, da_0
=&\int_{\partial \Omega_0}\left(\langle V^T, \nabla \widehat{v}_0 \rangle \frac{\partial \varphi_0}{\partial \nu_0} +v_0\frac{\partial \widehat{v}_0}{\partial \nu_0} \frac{\partial \varphi_0}{\partial \nu_0} 
+v_0\frac{\partial \varphi_0}{\partial\nu_0}\frac{\partial \widehat{v}_0}{\partial \nu_0}
\right)\,d\ell_0\\
&- \int_{\Omega_0}  \langle\nabla \langle\nabla \widehat{v}_0,\nabla\varphi_0\rangle,V\rangle\, da_0 +\lambda_1(\Omega_0)\int_{\Omega_0}  \langle V, \nabla \widehat{v}_0\rangle\varphi_0\,  da_0\\
&+\lambda_1(\Omega_0)\int_{\Omega_0} \langle V, \nabla \varphi_0\rangle\widehat{v}_0\, da_0\\
=&\,c\int_{\partial \Omega_0} \langle V^T, \nabla \widehat{v}_0 \rangle\, d\ell_0 +2c\int_{\partial \Omega_0}v_0 \frac{\partial \widehat{v}_0}{\partial \nu_0}\,d\ell_0\\
&-  \int_{\Omega_0}  \langle\nabla \langle\nabla \widehat{v}_0,\nabla\varphi_0\rangle,V\rangle\, da_0 +\lambda_1(\Omega_0)\int_{\Omega_0}  \langle V, \nabla \widehat{v}_0\rangle\varphi_0\,  da_0\\
&+\lambda_1(\Omega_0)\int_{\Omega_0} \langle V, \nabla \varphi_0\rangle\widehat{v}_0\, da_0.
\end{align*}
 
Next, let us deal with the last integral in the formula of $\left.\frac{d^2}{dt^2}\lambda_1(\Omega_t)\right|_{t=0}$ by applying integration by parts again and  using that $\tr_{\Omega_0}h = 2\dive V$.

We have:
\begin{align*}
\int_{\Omega_0} \langle \nabla \tr_{\Omega_0} h, \nabla \varphi_0\rangle\widehat{v}_0\,da_0
=&\int_{\Omega_0} \dive((\tr_{\Omega_0}h) \widehat{v}_0 \nabla \varphi_0)\,  da_0
-\int_{\Omega_0} (\tr_{\Omega_0}h) \dive(\widehat{v}_0 \nabla \varphi_0)\, da_0\\
=&\int_{\partial \Omega_0} (\tr_{\Omega_0} h) v_0 \frac{\partial\varphi_0}{\partial \nu_0}\,  d\ell_0
-2\int_{\Omega_0} \dive V  \dive(\widehat{v}_0 \nabla \varphi_0)\, da_0\\
=&\,c\int_{\partial\Omega_0}(\tr_{\partial\Omega_0}h) v_0\,d\ell_0 + c\int_{\partial\Omega_0}h(\nu_0,\nu_0) v_0\,d\ell_0\\
&-2\int_{\Omega_0} \dive V  \dive(\widehat{v}_0 \nabla \varphi_0)\, da_0,
\end{align*}
where in the last equality we have used that $\tr_{\Omega_0}h=\tr_{\partial\Omega_0}h+h(\nu_0,\nu_0)$. Now, since $V=\mathcal{L}_Vg$, we have that
\begin{align*}
\int_{\Omega_0} \langle \nabla \tr_{\Omega_0} h, \nabla \varphi_0\rangle\widehat{v}_0\,da_0=&\,2c\int_{\partial\Omega_0}(\dive_{\partial\Omega_0}V) v_0\,d\ell_0 + c\int_{\partial\Omega_0}h(\nu_0,\nu_0) v_0\,d\ell_0\\
&\hspace{-0.5cm}-2\int_{\Omega_0} \dive V (\langle\nabla\widehat{v}_0,\nabla\varphi_0\rangle+\widehat{v}_0\Delta\varphi_0)\, da_0
\end{align*}

By using that $\Delta\varphi_0+\lambda_1(\Omega_0)\varphi_0=0$, we obtain that
\begin{align*}
\int_{\Omega_0} \langle \nabla \tr_{\Omega_0} h, \nabla \varphi_0\rangle\widehat{v}_0\,da_0=&\,2c\int_{\partial\Omega_0}(\dive_{\partial\Omega_0}V) v_0\,d\ell_0 + c\int_{\partial\Omega_0}h(\nu_0,\nu_0) v_0\,d\ell_0\\
&-2\int_{\Omega_0} \dive V (\langle\nabla\widehat{v}_0,\nabla\varphi_0\rangle-\lambda_1(\Omega_0)\widehat{v}_0\varphi_0)\, da_0\\
=&\,2c\int_{\partial\Omega_0}(\dive_{\partial\Omega_0}V) v_0\,d\ell_0 + c\int_{\partial\Omega_0}h(\nu_0,\nu_0) v_0\,d\ell_0\\
&-2\int_{\Omega_0}\dive(\langle\nabla \widehat{v}_0,\nabla\varphi_0\rangle V)\,da_0+2\int_{\Omega_0} \langle V ,\nabla\langle\nabla\widehat{v}_0,\nabla\varphi_0\rangle\rangle\,da_0\\
&+2\lambda_1(\Omega_0)\int_{\Omega_0}(\dive V)\widehat{v}_0\varphi_0\, da_0.
\end{align*}

As $\dive_{\partial\Omega_0}V=\dive_{\partial\Omega_0}V^T+Hv_0$ on $\partial\Omega_0$, we obtain that

\begin{align*}
\int_{\Omega_0} \langle \nabla \tr_{\Omega_0} h, \nabla \varphi_0\rangle\widehat{v}_0\,da_0
&=2c\int_{\partial \Omega_0}v_0 \dive_{\partial\Omega_0}V^T\,  d\ell_0 + 2c\int_{\partial\Omega_0} H(v_0)^2\,d\ell_0\\
&\quad+c\int_{\partial\Omega_0}h(\nu_0,\nu_0)v_0\,d\ell_0-2\int_{\Omega_0} \dive(\langle \nabla \widehat{v}_0, \nabla \varphi_0\rangle V)\, da_0\\
&\quad+ 2\int_{\Omega_0} \langle V,\nabla\langle\nabla \widehat{v}_0,\nabla\varphi_0\rangle\rangle\,da_0+2\lambda_1(\Omega_0)\int_{\Omega_0} \dive(\widehat{v}_0\varphi_0 V)\,da_0
\\
&\quad-2\lambda_1(\Omega_0)\int_{\Omega_0} ( \langle V,\nabla\varphi_0\rangle \widehat{v}_0+
\langle V, \nabla\widehat{v}_0\rangle\varphi_0)\,da_0\\
&=-2c\int_{\partial \Omega_0}\langle \nabla_{\partial\Omega_0}v_0,V^T\rangle\,  d\ell_0 + 2c\int_{\partial\Omega_0} H(v_0)^2\,d\ell_0\\
&\quad+c\int_{\partial\Omega_0}h(\nu_0,\nu_0)v_0\,d\ell_0-2\int_{\Omega_0} \dive(\langle\nabla\widehat{v}_0,\nabla\varphi_0\rangle V)\,da_0 \\
&\quad+ 2\int_{\Omega_0} \langle V,\nabla\langle\nabla \widehat{v}_0,\nabla\varphi_0\rangle\rangle\,da_0 -2\lambda_1(\Omega_0)\int_{\Omega_0}  \langle V,\nabla\varphi_0\rangle \widehat{v}_0\,da_0\\
&\quad-2\lambda_1(\Omega_0)\int_{\Omega_0}
\langle V, \nabla\widehat{v}_0\rangle\varphi_0\,da_0.
\end{align*}

Thus, substituting into the formula above for $\left.\frac{d^2}{dt^2}\lambda_1(\Omega_t)\right|_{t=0}$ and canceling some terms, we have that
\begin{align*}
\left.\dfrac{d^2}{dt^2}\lambda_1(\Omega_t)\right|_{t=0}&= 4c^2\int_{\partial\Omega_0}v_0\dfrac{\partial\widehat{v}_0}{\partial\nu_0}\,d\ell_0+2c^2\int_{\partial \Omega_0} H(v_0)^2\,d\ell_0\\
&\quad -2c\int_{\Omega_0}\dive(\langle \nabla \widehat{v}_0,\nabla\varphi_0\rangle V)\,d\ell_0\\
&=2c^2\int_{\partial\Omega_0}\left(v_0\dfrac{\partial\widehat{v}_0}{\partial\nu_0}+H(v_0)^2\right)\,d\ell_0,
\end{align*}
and this finishes the proof.

\subsection*{Acknowledgments}
M.C. would like to express his gratitude to Princeton University for its hospitality, especially to Fernando Codá Marques, whose insightful questions inspired this work. This work was carried out while the authors were visiting ICTP - International Centre for Theoretical Physics, as associates. M.C. and I.N. acknowledge support from the ICTP through the Associates Programme (2022-2027 and 2019-2024, respectively). The authors are grateful to the Institute and to Claudio Arezzo for their hospitality. 
We would also like to thank Lucas Ambrozio for suggesting Theorem \ref{topbound2}, and Levi Lima for bringing references \cite{Shimakura2} and \cite{JorgedeLima} to our attention.

The authors were partially supported by the Brazilian National Council for Scientific and Technological Development (CNPq), Grant 405468/2021-0. M.C. was also supported by CNPq under Grant 11136/2023-0 and by the Foundation for Research Support of the State of Alagoas (FAPEAL) under Grant 60030.0000000323/2023.

\bibliographystyle{amsplain}
\bibliography{references} 

\end{document}